\documentclass{amsart}
\usepackage{amsmath,amsthm,amssymb}
\usepackage[colorlinks]{hyperref}
\usepackage{cleveref}
\usepackage{biblatex}
\usepackage{microtype}
\usepackage{quiver}

\DeclareMathOperator{\nm}{nm}
\DeclareMathOperator{\tr}{tr}
\DeclareMathOperator{\res}{res}
\DeclareMathOperator{\conj}{conj}
\DeclareMathOperator{\Spec}{Spec}

\newcommand{\uA}{\underline{A}}
\newcommand{\Z}{\mathbb{Z}}
\newcommand{\Tamb}{\mathsf{Tamb}}

\mathchardef\mhyphen="2D

\newtheorem{thm}{Theorem}
\newtheorem{lem}{Lemma}
\newtheorem{cor}{Corollary}
\newtheorem{prop}{Proposition}
\newtheorem{question}{Question}
\theoremstyle{definition}
\newtheorem{example}{Example}

\title{Inverting Integers in Tambara Functors}
\author[B.\ Spitz]{Ben Spitz}
\address{Department of Mathematics, Indiana University, Bloomington, IN, USA}
\email{bespitz@iu.edu}

\begin{document}

\begin{abstract}
    Let $G$ be a finite group, and $k$ an integer. In this note, we show that for any $G$-Tambara functor $T$ and any subgroups $H_1, H_2 \leq G$, $k$ is a unit in $T(G/H_1)$ if and only if $k$ is a unit in $T(G/H_2)$. In other words, one may speak unambiguously of the localization $T[1/k]$.  As a consequence, the norm functors $N_H^G$ commute with inverting $k$.
\end{abstract}

\maketitle

\section{Introduction}\label{sec:introduction}

Tambara functors are equivariant generalizations of commutative rings, in the following sense: for each finite group $G$, there is a notion of \emph{$G$-Tambara functors}, and this notion coincides with that of commutative rings when $G$ is the trivial group.

The data of a $G$-Tambara functor $T$ consists of a collection of commutative rings $T(G/H)$, one for each subgroup $H$ of $G$, together with:
\begin{enumerate}
    \item for each inclusion $H \leq K$ of subgroups of $G$:
    \begin{enumerate}
        \item a ring homomorphism $\res_H^K : T(G/K) \to T(G/H)$,
        \item a homomorphism of additive groups $\tr_H^K : T(G/H) \to T(G/K)$, and
        \item a homomorphism of multiplicative monoids $\nm_H^K : T(G/H) \to T(G/K)$;
    \end{enumerate}
    \item for each element $g \in G$ and each subgroup $H \leq G$, an isomorphism of rings $\conj_{g,H} : T(G/H) \to T(G/gHg^{-1})$.
\end{enumerate}
These data must satisfy a long list of axioms -- we refer the reader to \cite{Tambara} for the original definition of Tambara functors, and \cite{Strickland} for a thorough introduction to the theory of Tambara functors.

Fixing a finite group $G$, the category of $G$-Tambara functors is a multi-sorted variety in the sense of universal algebra\footnote{Indeed, the category of $G$-Tambara functors is best defined as the category of algebras for a certain multi-sorted Lawvere theory. We refer the reader to \cite[Chapter 3]{Adamek-Rosicky} for an introduction to the theory of multi-sorted varieties.} \cite[Proof of Proposition 4.3]{AlgClosed}. In particular, one can perform all sorts of universal constructions: quotients by ideals, localizations by inverting elements, direct products, etc. are well-defined notions in the category of $G$-Tambara functors.

This note addresses the following question, which we answer in the affirmative:
\begin{question}\label{question}
    Let $T$ be a $G$-Tambara functor and $k$ an integer. One may either:
    \begin{enumerate}
        \item localize $T$ by inverting the element $k$ in the ring $T(G/e)$, or
        \item localize $T$ by inverting the element $k$ in the ring $T(G/G)$.
    \end{enumerate}
    Do these produce the same result, up to isomorphism?
\end{question}

There is an a priori relationship between these two constructions: inverting $k \in T(G/G)$ is a ``more forceful'' operation than inverting $k \in T(G/e)$, similar to how inverting the element $6$ in the ring $\Z$ is more forceful than inverting $2$ (i.e. $2$ is a unit in $\Z[1/6]$). Consequently, there is a natural map from the localization (1) to the localization (2), and the question more precisely asks if this natural map is an isomorphism. We will show that this is the case via a simple argument in commutative algebra (\Cref{main theorem}). For any specific choice of $G$, $T$, and $k$, one may prove this directly without any input from commutative algebra. However, the necessary argument will depend on the structure of $G$ and the prime factorization of $k$, as demonstrated in \Cref{example:by hand for G=C_2}. The argument we give has the advantage of being completely uniform in $G$, $T$, and $k$.

A related but orthogonal question is addressed by Hill and Hopkins~\cite{HillHopkins}, who studied the behavior of genuine $G$-spectra under localization by inverting an integer. Just as $E_\infty$ ring spectra are structures in the stable homotopy category which yield commutative rings after applying $\pi_0$, there are structures known as \emph{$G{\mhyphen}E_\infty$ ring spectra} in the $G$-equivariant stable homotopy category which yield Tambara functors after applying $\pi_0$. The work of Hill and Hopkins shows that one may perform such a localization of $G{\mhyphen}E_\infty$ ring spectra either by inverting an integer $k$ as an element of the $G{\mhyphen}E_\infty$-ring $R$ or by inverting $k$ as an endomorphism of $R$ (viewed as a module over itself). Fundamentally, their work is concerned with the particular construction of the localization (i.e., identifying exactly which elements must be inverted by the localization map $R \to R[k^{-1}]$, or in other words asking how one may concretely describe the localization $R[k^{-1}]$).

At the level of Tambara functors, this was also addressed also by Nakaoka~\cite{Nakaoka}, who gave a concrete construction of localizations of Tambara functors. This is often extremely useful in particular computations; however, abstractly, one need not consider any particular construction -- as mentioned above, it is purely formal that $G$-Tambara functors admit universal constructions such as localization. It is from this abstract point of view that we approach the question of interest.

Recent progress in the theory of equivariant algebra (e.g. the computation of prime-ideal spectra~\cite{GhostsI,GhostsII,WisdomStratification} and the classification of the equivariant analogues of algebraically closed fields~\cite{AlgClosed}) has been enabled by the development of a basic theory of commutative algebra for Tambara functors~\cite{NakaokaPrimes,Nakaoka,Sun,Etale,WisdomEtale,CpnTambaraFields}. Just as ordinary commutative algebra is an indispensable tool for computations in ordinary stable homotopy theory (whenever one studies $E_\infty$ rings), the commutative algebra of Tamabra functors has applications to equivariant stable homotopy theory (whenever one studies $G{\mhyphen}E_\infty$ rings). The behavior of localizations of Tambara functors is of central importance to this theory, and plays a crucial role in some existing computations~\cite{HillMehrleQuigley,Bohme}. It is with an eye towards this program that this note is written.

\subsection*{Acknowledgments}

The author thanks Ayelet Lindenstrauss for prompting this line of inquiry and David Chan for helpful conversations. The author also thanks an anonymous referee for valuable comments and corrections.

\section{Primer on Tambara Functors}

In this section we review the essential information about the category of Tambara functors which is needed for the proof of the main theorem.

Fix a finite group $G$. There is a category $\Tamb_G$ of \emph{$G$-Tambara functors}, whose objects are $G$-Tambara functors (as described in \Cref{sec:introduction}). A morphism $f : T \to T'$ of $G$-Tambara functors is a collection $\{f_H : T(G/H) \to T'(G/H)\}_{H \leq G}$ of ring homomorphisms which commute with the maps $\res_H^K$, $\tr_H^K$, $\nm_H^K$, and $\conj_{g,H}$ individually. For example, a morphism $f : T \to T'$ of $G$-Tambara functors yields in particular a commutative square
\[\begin{tikzcd}[ampersand replacement=\&]
	{T(G/G)} \& {T'(G/G)} \\
	{T(G/e)} \& {T'(G/e)}
	\arrow["{f_G}", from=1-1, to=1-2]
	\arrow["{\res^G_e}"', from=1-1, to=2-1]
	\arrow["{\res^G_e}", from=1-2, to=2-2]
	\arrow["{f_e}"', from=2-1, to=2-2]
\end{tikzcd}\]
in the category of commutative rings.

\subsection{The Burnside Tambara Functor}

The category of $G$-Tambara functors has an initial object, called the \emph{Burnside Tambara functor}, denoted $\uA$. For what follows, we need only record the following facts about the Burnside Tambara functor:

\begin{enumerate}
    \item $\uA(G/G)$ is canonically isomorphic to the \emph{Burnside ring} $A(G)$ of $G$ (the Grothendieck ring of the category of finite $G$-sets\footnote{By $G$-set we will always mean \emph{left} $G$-set.}). That is, the elements of $A(G)$ are formal differences of isomorphism classes of finite $G$-sets, and the addition and multiplication operations of $A(G)$ are determined by $[X] + [Y] := [X \amalg Y]$ and $[X] \cdot [Y] = [X \times Y]$ for all finite $G$-sets $X, Y$.
    \item $\uA(G/e)$ is canonically isomorphic to the ring of integers $\Z$.
    \item The function $\nm_e^G : \uA(G/e) \cong \Z \to A(G) \cong \uA(G/G)$ sends each natural number $k$ to the isomorphism class of the $G$-set of \emph{all} functions $G \to \{1, \dots, k\}$. The action of $G$ on this set of functions is by precomposition with the right action of $G$ on itself by multiplication, i.e.\ $(g \cdot f)(g') = f(g'g)$.
\end{enumerate}

\subsection{Localization}

Given any $G$-Tambara functor $T$, any subgroup $H \leq G$, and any element $x \in T(G/H)$, there exists an initial Tambara functor $T'$ receiving a map $\phi$ from $T$ such that $\phi_H(x)$ is a unit in $T'(G/H)$. This initial Tambara functor is called the \emph{localization} of $T$ by $x$, and is denoted $T[1/x]$. An explicit construction of $T[1/x]$ is given by Nakaoka~\cite{Nakaoka}.

We note that these localizations may exhibit behavior which is unexpected to those familiar with commutative algebra. For example, there can exist elements $x \in T(G/H)$ which are not nilpotent but yield $T[1/x] = 0$ (see \Cref{example:zero localization from non-nilpotent element}). Nonetheless, these localizations are well-defined, i.e. there always exists an object (unique up to unique isomorphism) with the above universal property.

This note is concerned with the localizations of Tambara functors $T$ by integers $k$. Since each ring $T(G/H)$ contains a (unique) homomorphic image of $\Z$, we must take care (a priori) to describe at which level a localization of $T$ by an integer $k$ is taking place. For a subgroup $H$ of $G$, we will write $T[1/k_H]$ to denote the localization of $T$ by the element $k \in T(G/H)$.

For any inclusion $H \leq K$ of subgroups of $G$, the structure map $\res^K_H$ in a $G$-Tambara functor is a ring homomorphism; therefore $\res^K_H$ must act as the identity on integers and must send units to units. Since $k \in T[1/k_K](G/K)^\times$ by definition, we also have $k \in T[1/k_K](G/H)^\times$. Now the structure morphism $\varphi : T \to T[1/k_K]$ has the property that $\varphi_H(k) = k \in T[1/k_K](G/H)^\times$, so by the universal property of localization we obtain a canonical morphism $T[1/k_H] \to T[1/k_K]$. In \Cref{sec:main result} we will show that this is an isomorphism.

\section{The Ghost Map}

For each subgroup $H \leq G$, there is a ring homomorphism $\chi^H : A(G) \to \Z$ determined by $[X] \mapsto \#X^H$, i.e. each $G$-set is sent to its number of $H$-fixed points. Taking all of these homomorphisms together yields the \emph{ghost map}\footnote{Also sometimes known as the \emph{mark(s) homomorphism}}
\[\chi : A(G) \to \prod_{H \leq G} \Z.\]
A key feature of the Burnside ring is that the ghost map $\chi$ is injective (a fact most likely first observed by W.\ Burnside himself). This map was used to great effect by Dress~\cite{Dress}, who further observed that, since $\prod_{H \leq G} \Z$ is finitely generated as an abelian group, $\chi$ is an integral extension of rings\footnote{To be precise, Dress considers a slightly different homomorphism (taking only the maps $\chi^H$ as $H$ ranges over a set of representatives of the conjugacy classes of subgroups of $G$). This makes no important difference to the discussion in this section.}. Therefore,
\[\chi^* : \coprod_{H \leq G} \Spec \Z \cong \Spec \left(\prod_{H \leq G} \Z\right) \longrightarrow \Spec A(G)\]
is surjective. This immediately yields a description of (the underlying set of) $\Spec A(G)$.

\begin{prop}[Dress]\label{prop:Dress}
    Let $\mathfrak{p}$ be a prime ideal of $A(G)$. Then there exists a subgroup $H \leq G$ and a prime ideal $\mathfrak{q} \in \Spec \Z$ such that
    \[\mathfrak{p} = \ker(A(G) \xrightarrow{\chi^H} \Z \to \Z/\mathfrak{q}).\]
\end{prop}

The ghost map also gives better understanding of the map $\nm_e^G$ in the Burnside Tambara functor.

\begin{lem}\label{lem:marks of norm}
    Let $k$ be a natural number, and let $H$ be a subgroup of $G$. Then
    \[\chi^H(\nm_e^G(k)) = k^{[G:H]}.\]
\end{lem}
\begin{proof}
    We must count the number of $H$-fixed points in the $G$-set of all functions $G \to \{1, \dots, k\}$. By definition, an $H$-fixed point in this $G$-set is a function $f$ such that
    \[f(gh) = f(g)\]
    for all $g \in G$ and all $h \in H$. In other words, an $H$-fixed point is a function $G \to \{1, \dots, k\}$ which is constant on left $H$-cosets. There are $[G:H]$-many left cosets of $H$ in $G$, and therefore $k^{[G:H]}$-many such functions.
\end{proof}

\section{Main Results}\label{sec:main result}

We now answer \Cref{question} in the affirmative. To restate, we wish to show that for any $G$-Tambara functor $T$, the natural map $T[1/k_e] \to T[1/k_G]$ is an isomorphism, where $T[1/k_e]$ denotes the localization of $T$ by the element $k \in T(G/e)$ and $T[1/k_G]$ denotes the localization of $T$ by the element $k \in T(G/G)$. For $H$ any subgroup of $G$, the localization $T[1/k_H]$ is intermediate between $T[1/k_e]$ and $T[1/k_G]$ -- thus, it will follow that the localizations $T[1/k_H]$ are all canonically isomorphic (\Cref{cor:main}).

Consequently, once this is proved, there will be no need to use such cumbersome notation; one may simply write $T[1/k]$ without any fear of ambiguity. By the universal property of localization, it is equivalent to prove the following theorem.

\begin{thm}\label{main theorem}
    Let $G$ be a finite group, let $T$ be a $G$-Tambara functor, and let $k$ be an integer. The following are equivalent:
    \begin{enumerate}
        \item $k \in T(G/G)^\times$;
        \item $k \in T(G/H)^\times$ for all subgroups $H \leq G$;
        \item $k \in T(G/e)^\times$.
    \end{enumerate}
\end{thm}

We note that the proof below uses in an essential way the structure map $\nm_e^G$ of a $G$-Tambara functor. For Green functors (which differ only from Tambara functors in that they do not have the $\nm$ structure maps) the theorem is false (see \Cref{example:main theorem fails for Green functors}).

\begin{proof}
    The implication (1) $\implies$ (2) is trivial, since we have ring homomorphisms $\res_H^G : T(G/G) \to T(G/H)$ for all subgroups $H \leq G$. The implication (2) $\implies$ (3) is also trivial. Only (3) $\implies$ (1) remains to be shown.

    First, we may assume without of loss of generality that $k \geq 0$, since the group of units of a ring is closed under multiplication by $-1$.

    Now suppose $k \in T(G/e)^\times$. Since $\nm_e^G$ is multiplicative, we have $\nm_e^G(k) \in T(G/G)^\times$. The unique map $\varphi : \uA \to T$ gives a homomorphism $\varphi_G : A(G) \to T(G/G)$ such that
    \[\varphi_G(\nm_e^G(k)) = \nm_e^G(\varphi_e(k)) = \nm_e^G(k)\]
    is a unit in $T(G/G)$. Thus, we have a homomorphism $A(G)[1/\nm_e^G(k)] \to T(G/G),$
    and it suffices to show that $k \in A(G)[1/\nm_e^G(k)]^\times$.

    Suppose for contradiction that $k \notin A(G)[1/\nm_e^G(k)]^\times$. Then there is a maximal ideal $\mathfrak{m}$ of $A(G)[1/\nm_e^G(k)]$ such that $k \in \mathfrak{m}$. Such an ideal necessarily has the form \[\mathfrak{m} = \left\{\frac{x}{\nm_e^G(k)^s} : s \in \mathbb{N}, x \in \mathfrak{p}\right\}\] for some prime ideal $\mathfrak{p} \in \operatorname{Spec} A(G)$ such that $\nm_e^G(k) \notin \mathfrak{p}$.

    By \Cref{prop:Dress}, we have $\mathfrak{p} = \ker(A(G) \xrightarrow{\chi^H} \Z \to \Z/\mathfrak{q})$ for some subgroup $H \leq G$ and some prime ideal $\mathfrak{q} \in \operatorname{Spec}(\Z)$. Now $k \in \mathfrak{m}$ implies
    \[k \nm_e^G(k)^s = x \nm_e^G(k)^t\]
    in $A(G)$ for some $x \in \mathfrak{p}$ and $s,t \in \mathbb{N}$. Reducing modulo $\mathfrak{p}$, we have
    \[k \nm_e^G(k)^s = 0\]
    in $A(G)/\mathfrak{p}$. Since $\nm_e^G(k) \notin \mathfrak{p}$ and $A(G)/\mathfrak{p}$ is a domain, this implies $k \in \mathfrak{p}$. In other words,
    \[k = \chi^H(k) \in \mathfrak{q}.\]
    However, this yields (by \Cref{lem:marks of norm})
    \[\chi^H(\nm_e^G(k)) = k^{[G:H]} \in \mathfrak{q},\]
    so $\nm_e^G(k) \in \mathfrak{p}$, a contradiction.
\end{proof}

\begin{cor}\label{cor:main}
Let $G$ be a finite group. Let $T$ be a $G$-Tambara functor, let $H$ be a subgroup of $G$, and let $k$ be an integer. There are canonical isomorphisms $T[1/k_e] \to T[1/k_H] \to T[1/k_G]$, and thus one may speak unambiguously of the localization $T[1/k]$.
\end{cor}

This result simplifies many arguments in equivariant algebra. To conclude our discussion, we give one example of such an argument. For any inclusion of finite groups $H \leq G$, there is a forgetful functor $R_H^G : \Tamb_G \to \Tamb_H$ (called \emph{restriction}), given levelwise on objects by $(R_H^G T)(H/K) = T(G/K)$. This functor has a left adjoint $N_H^G : \Tamb_H \to \Tamb_G$ (called \emph{norm})~\cite{HillMazur,Hoyer}, which is the algebraic analogue of the Hill-Hopkins-Ravenel norm~\cite{HHR}. The norm functor $N_H^G$ can be computed by a coend
\[(N_H^G T)(G/K) = \int^{X \in \mathsf{set}_H} \mathcal{A}_G(\mathsf{set}_H(G,X), G/K) \times T(X),\]
which we display only to persuade the reader that dealing with $N_H^G$ this way is typically painful -- indeed, this coend formula describes only the underlying set of $(N_H^G T)(G/K)$; defining the ring operations on this set is even more tedious. However, one special case is easy to compute: by a straightforward abstract argument with universal properties, one finds that $(N_H^G T)(G/e) \cong \bigotimes_{[G : H]} T(H/e)$. A more general fact is recorded by Hoyer~\cite[Theorem 2.5.1]{Hoyer}. We immediately obtain:

\begin{cor}
    Let $H \leq G$ be an inclusion of groups and let $k$ be an integer. Let $T$ be an $H$-Tambara functor in which $k$ is invertible at some level. Then $k$ is invertible in $N_H^G(T)$ at all levels. More generally, $N_H^G$ commutes with inverting $k$, i.e. there is a natural isomorphism $N_H^G(T[1/k]) \cong N_H^G(T)[1/k]$ (with no assumption on $T$).
\end{cor}

\appendix

\section{Supplementary Examples}

\begin{example}\label{example:by hand for G=C_2}
    Here we give an example of how one might prove a specific case of \Cref{main theorem} without input from commutative algebra.

    Fix a prime number $p$, and let $G = C_p$ be the cyclic group of order $p$. Then
    \[\uA \cong \begin{tikzcd}
    	{\frac{\Z[t]}{(t^2-pt)}} \\
    	\\
    	\Z
    	\arrow["{\res_e^{C_p}}"{description}, from=1-1, to=3-1]
    	\arrow["{\tr_e^{C_p}}", shift left=5, from=3-1, to=1-1]
    	\arrow["{\nm_e^{C_p}}"', shift right=5, from=3-1, to=1-1]
    \end{tikzcd}\]
    where
    \begin{align*}
        \nm_e^{C_p}(x) &= x + \frac{x^p-x}{p} t, \\
        \res_e^{C_p}(t) &= p, \\
        \tr_e^{C_p}(1) &= t.
    \end{align*}
    Now let $k$ be any integer. We wish to show that $k$ is a unit in $\uA[1/k_e](C_p/C_p)$, i.e. inverting $k \in \uA(C_p/e)$ also forces $k$ to become inverted in $\uA(C_p/C_p)$. Since $\nm_e^{C_p}$ sends units to units, we know that $\nm_e^{C_p}(k) = k + \frac{k^p-k}{p} t$ will be a unit in $\uA[1/k_e](C_p/C_p)$. If $p$ does not divide $k$, then $k$ divides $\frac{k^p - k}{p} = k \frac{k^{p-1}-1}{p}$. In this case, we have that $k(1 + \frac{k^{p-1}-1}{p}t)$ is a unit in $\uA[1/k_e](C_p/C_p)$, and thus $k \in \uA[1/k_e](C_p/C_p)^\times$. Otherwise, if $p$ divides $k$, we have that
    \begin{multline*}
        \nm_e^{C_p}(k)^2
        = \left(k + \frac{k}{p}(k^{p-1}-1)t\right)^2 \\
        = k^2 + 2 k \frac{k}{p} (k^{p-1}-1)t + p \left(\frac{k}{p}\right)^2 \left(k^{p-1}-1\right)^2 t \\
        = k \left(k + 2 \frac{k}{p} (k^{p-1}-1) t + \frac{k}{p} (k^{p-1}-1)^2 t\right)
    \end{multline*}
    is a unit in $\uA[1/k_e](C_p/C_p)$, and so again $k \in \uA[1/k_e](C_p/C_p)^\times$.
\end{example}

\begin{example}\label{example:zero localization from non-nilpotent element}
    Here we give an example of a Tambara functor $T$ and an element $x$ such that $x$ is not nilpotent and $T[1/x]$ is the zero Tambara functor.

    Fix a prime number $p$, and let $G = C_p$ be the cyclic group of order $p$. Then $\underline{A}(C_p/C_p) \cong A(C_p) \cong \Z[t]/(t^2-pt)$, where the element $t$ corresponds to the isomorphism class of the free orbit $[C_p/e]$, and $1$ corresponds to the isomorphism class of the fixed orbit $[C_p/C_p]$. The element $t-p$ is not nilpotent in this ring, because its image under $\chi^{C_p} : A(C_p) \to \mathbb{Z}$ is the nonzero integer $-p$. However, $\res_e^{C_p}(t-p) = 0$, so $\underline{A}[1/(t-p)](C_p/e)$ is the zero ring. Since $C_p$-Tambara functors always satisfy $\nm_e^{C_p}(0) = 0$ and $\nm_e^{C_p}(1) = 1$, we conclude that $\underline{A}[1/(t-p)](C_p/C_p)$ is also the zero ring, and thus $\underline{A}[1/(t-p)]$ is the zero Tambara functor.
\end{example}

\begin{example}\label{example:main theorem fails for Green functors}
    Here we give an example of the failure of \Cref{main theorem} for Green functors, demonstrating that the result relies on the existence of the norm maps in Tambara functors.

    Fix a prime number $p$, and let $G = C_p$ be the cyclic group of order $p$. Consider the Green functor
    \[T = \begin{tikzcd}[ampersand replacement=\&]
    	{\Z[t, t/p, t/p^2, \dots]/(t^2-pt)} \\
    	{\Z[1/p]}
    	\arrow["{\res_e^{C_p}}", shift left=2, from=1-1, to=2-1]
    	\arrow["{\tr_e^{C_p}}", shift left=2, from=2-1, to=1-1]
    \end{tikzcd}\]
    where the restriction and transfer maps are determined by
    \begin{align*}
        \res_e^{C_p}(t/p^n) &= p^{1-n}, \\
        \tr_e^{C_p}(1/p^n) &= t/p^n.
    \end{align*}
    There is a ring homomorphism $\varphi : T(C_p/C_p) \to \mathbb{Z}$ determined by $\varphi(t) = 0$; since $p \notin \mathbb{Z}^\times$ we have also that $p \notin T(C_p/C_p)^\times$. Thus, $T$ satisfies $p \in T(C_p/e)^\times$ and $p \notin T(C_p/C_p)^\times$. In particular, \Cref{main theorem} fails for Green functors.
\end{example}

\printbibliography

\end{document}